\documentclass[11pt]{article}

\usepackage{fullpage}
\usepackage{amsmath, amsthm, amsfonts, amssymb, amstext, mathrsfs, enumerate}
\usepackage{graphicx, ragged2e, lscape, framed, xcolor}
\usepackage{subfiles}

\theoremstyle{plain}
\newtheorem{theorem}{Theorem}[section]

\newtheorem{conjecture}[theorem]{Conjecture}
\newtheorem{corollary}[theorem]{Corollary}

\newtheorem*{remark}{Remark}

\numberwithin{equation}{section}
\allowdisplaybreaks

\newcommand{\affl}[3]{\noindent #1, Email: {\tt #2}\\ \textsc{#3}\\[1.5pt]}

\usepackage[pagebackref]{hyperref}
\hypersetup{
	colorlinks=true,
    urlcolor=purple,
	linkcolor=purple,
    citecolor=purple,
}

\DeclareMathOperator{\supp}{supp}

\title{\textbf{Localized Tur\'{a}n-type inequalities for $Q$-index}}
\author{M. Rajesh Kannan, Hitesh Kumar, Shivaramakrishna Pragada}
\date{}

\begin{document}
\maketitle
\begin{abstract}
For a connected graph \(G\), let $q(G)$ denote the $Q$-index of $G$, i.e., the largest eigenvalue of its signless Laplacian matrix. Abreu and Nikiforov (2013) showed that
\[
    q(G) \leq 2n\left(1-\frac{1}{\omega(G)}\right), 
\]
where $\omega(G)$ denotes the clique number of $G$. We first give a short algebraic proof of this result. For a vertex $v\in V(G)$, let \(c(v)\) denote the order of the largest clique of \(G\) containing \(v\). Our main result is the following vertex localized bound that refines the result of Abreu and Nikiforov:
\[
    q(G) \leq
    2\sum_{v\in V(G)}\left(1-\frac{1}{c(v)}\right).
\]
Equality holds precisely for complete bipartite graphs when
\(\omega(G)=2\), and for regular complete \(\omega(G)\)-partite graphs
when \(\omega(G)\geq 3\). As a consequence, we also obtain an analogous localized inequality for the $A_\alpha$-matrix of $G$. Finally, we generalize the above localized inequality to vertex-weighted signed graphs. This contributes to the localization program for spectral Tur\'{a}n-type results. 
\end{abstract}



\section{Introduction}

We use standard graph theory notation throughout the paper. Let $G = (V(G), E(G))$ be a finite simple graph of order $n$ and size $m$. Denote the \emph{clique number} of $G$ by $\omega(G)$(or simply $\omega$ when the underlying graph is clear from the context). Let $\deg(v)$ denote the degree of vertex $v$ in $G$. We assume that $G$ is connected throughout. For an edge $e\in E$ (or a vertex $v\in V$), we will denote by $c(e)$ (resp. $c(v)$) the size of the largest clique in $G$ containing the edge $e$ (resp. the vertex $v$). The \textit{adjacency matrix} of $G$ is an $n\times n$ matrix $A(G) = [a_{uv}]$, where $a_{uv} = 1$ if $u$ and $v$ are adjacent, and $0$ otherwise. The \emph{degree matrix} $D(G)$ is a diagonal matrix with degrees of vertices of $G$ on the diagonal. The matrices 
\[L(G) := D(G) - A(G)\quad \text{and} \quad Q(G) := D(G) + A(G)\]
are called the \emph{Laplacian} and the \emph{signless Laplacian matrix} of $G$, respectively. A real symmetric matrix $M$ of order $n$ has real eigenvalues, which we enumerate as   
\[\lambda_1(M) \geq  \cdots \geq \lambda_n(M).\] 
We let $\lambda_1(G):= \lambda_1(A(G))$. Define $q(G): = \lambda_1(Q(G))$ and call it the $Q$-index of $G$. 

There is a vast literature on spectral Tur\'{a}n type inequalities for $\lambda_1(G)$ \cite{Li_Liu_Zhang_2025_stone_simonovits,Nikiforov_2002, Nikiforov_2009, Wilf_1986} and $q(G)$ \cite{Abreu_Nikiforov_2013, He_Jin_Zhang_2013, Zheng_Li_Su_2026, Zheng_Li_Fan_2026, Zheng_Li_Li_2026_color_critical}. Recently, there has been an interest in vertex and edge localization of classical inequalities including but not limited to the classic Tur\'{a}n's Theorem, see \cite{Malec_Tompkins_2023} and its follow-up papers. The spectral Tur\'{a}n theorems for $\lambda_1(G)$ were localized by Lele Liu and Bo Ning \cite{Liu_Ning_2025_weighted, Liu_Ning_2026}. Precisely, they proved the following.

\begin{theorem}[Edge-localized Spectral Tur\'{a}n's Theorem \cite{Liu_Ning_2025_weighted, Liu_Ning_2026}] \label{thm:spectral_edge_localization_Turan}
For any connected graph $G$, 
\begin{equation*}
   \lambda_1^2(G) \le \sum_{e\in E(G)} 2\left(1-\frac{1}{c(e)}\right). 
\end{equation*}
Equality holds if and only if $G$ is a complete bipartite graph when $\omega(G)=2$ or a regular complete $\omega$-partite graph when $\omega(G)\ge 3$. 
\end{theorem}

\begin{theorem}[Vertex-localized Spectral Tur\'{a}n's Theorem \cite{Liu_Ning_2025_weighted}] \label{thm:spectral_vertex_localization_Turan}
For any connected graph $G$,
\begin{equation*}
   \lambda_1(G) \le \sum_{v\in V(G)} \left(1-\frac{1}{c(v)}\right). 
\end{equation*}
Equality holds if and only if $G$ is a regular complete multi-partite graph.
\end{theorem}

Kannan, Kumar, and Pragada \cite{Kannan_Kumar_Pragada_2025} investigated several new localized inequalities. Recently, Feng Liu, Shuang Sun, Yan Wang, and Qi Wu \cite{Liu_Sun_Wang_Wu_2026} established a localized inequality for $\lambda_1(G)$ involving walks that generalizes several past results. In this paper, we focus on localized inequalities for the $Q$-index $q(G)$. 

It is well-known that 
\begin{equation}\label{eq:Q_index_lambda_1}
    2\lambda_1(G)\le q(G)\le 2\Delta(G).
\end{equation}

Abreu and Nikiforov \cite{Abreu_Nikiforov_2013} established the following bound for $q(G)$. In view of \eqref{eq:Q_index_lambda_1}, this inequality generalizes Wilf's inequality \cite{Wilf_1986} and therefore the classical Tur\'{a}n's Theorem.

\begin{theorem}[\cite{Abreu_Nikiforov_2013}]\label{thm:signless_vertex}
For any connected graph $G$, we have
    \[q(G) \leq 2n \left(1 - \frac{1}{\omega(G)}\right).\]
Equality holds if and only if $G$ is a complete bipartite graph when $\omega=2$ or a regular complete $\omega$-partite graph when $\omega\ge 3$. 
\end{theorem}

Abreu and Nikiforov \cite{Abreu_Nikiforov_2013} established Theorem \ref{thm:signless_vertex} using an inductive argument. The equality case was characterized by Bian He, Ya-Lei Jin, and Xiao-Dong Zhang \cite{He_Jin_Zhang_2013}. We first give an alternative short algebraic proof of Theorem \ref{thm:signless_vertex} in Section \ref{sec:Nikiforov_short_proof}. 

Our main result is the following vertex localized inequality for $Q$-index.

\begin{theorem}\label{thm:signless_vertex_local}
For any connected graph $G$, we have 
    \[ q(G) \leq 2\sum_{v \in V}\left(1 - \frac{1}{c(v)}\right).\] 
Equality holds if and only if $G$ is a complete bipartite graph when $\omega=2$ or a regular complete $\omega$-partite graph when $\omega\ge 3$.     
\end{theorem}

We prove this by first translating $Q(G)$ to a weighted adjacency matrix and then using a weighted version of the Motzkin-Strauss Theorem. 

Recall that for a given graph $G$ and $\alpha\in [0,1]$, 
\[ A_\alpha(G):= \alpha D(G) + (1-\alpha)A(G).\]
We also note the following corollary for the $A_\alpha$ matrix. 

\begin{corollary}\label{cor:A_alpha_local}
    For any connected graph $G$ and $0\le \alpha\le 1/2$, we have 
    \[ \lambda_1(A_\alpha(G)) \le \sum_{v \in V}\left(1 - \frac{1}{c(v)}\right). \]
\end{corollary}

The proofs of Theorem \ref{thm:signless_vertex_local} and Corollary \ref{cor:A_alpha_local} are given in Section \ref{sec:localized_Q_index}. 

Finally, we take a significant step further and generalize Theorem \ref{thm:signless_vertex_local} to weighted signed graphs. We state and prove our results on vertex-weighted signed graphs in Section \ref{sec:weighted_signed_graph}. We conclude with some open problems in Section \ref{sec:open_problems}.

\section{A short proof of the Abreu-Nikiforov bound}
\label{sec:Nikiforov_short_proof}

Here, we give a short algebraic proof of Theorem \ref{thm:signless_vertex}. We first recall the Motzkin-Strauss Theorem. Let $G$ be a graph on $n$ vertices. Consider the standard simplex given by 
\[S = \left\{x \in \mathbb{R}^n : \sum_{v \in V(G)} x_v = 1, x_v \geq 0, v\in V(G)\right\}.\]
For a vector $x\in \mathbb{R}^n$, the \emph{support} of $x$ is given by $\supp(x) = \{v\in V(G) : x_v\neq 0\}$. 

\begin{theorem}[Motzkin-Strauss Theorem \cite{Motzkin_Straus_1965}]\label{thm:Motzkin_Strauss}
Let $G$  be a graph with adjacency matrix $A(G)$ and clique number $\omega$. For any $x \in S$, we have  
	\[x^TA(G)x \leq 1- \frac{1}{\omega}.\]
    Equality holds if and only if $\supp(x)$ induces a complete $\omega$-partite graph.
\end{theorem}

\begin{proof}[Proof of Theorem \ref{thm:signless_vertex}]

Let $x$ be the unit Perron-eigenvector for the $Q$-index $q(G)$. Since $Q(G) = L(G) + 2A(G)$, we have
\[ q(G) = x^TL(G)x + 2x^TA(G)x.\]
Now, 
\[
    x^TL(G)x = \sum_{uv\in E(G)}(x_u - x_v)^2 \le \sum_{uv\in E(K_n)}(x_u - x_v)^2 = n - \left(\sum_{v\in V(G)}x_v\right)^2. 
\]

Again, Theorem \ref{thm:Motzkin_Strauss} implies 
\begin{equation}\label{eq:AN_1}
    x^TA(G)x \le \left(\sum_{v\in V(G)}x_v\right)^2 \left(1 - \frac{1}{\omega(G)}\right).
\end{equation}

On combining
\[ q(G) \le n + \left(\sum_{v\in V(G)}x_v\right)^2\left(1 - \frac{2}{\omega(G)}\right).\]
Using the Cauchy-Schwarz inequality, we have 
\begin{equation}\label{eq:AN_2}
    \left(\sum_{v\in V(G)}x_v\right)^2 \le n\left(\sum_{v\in V(G)}x_v^2\right) = n.
\end{equation}
We conclude that 
\[ q(G) \le n + n\left(1 - \frac{2}{\omega(G)}\right) = 2n\left(1 - \frac{1}{\omega(G)}\right).\]

It is clear that equality holds for regular $\omega$-partite graphs when $\omega \ge 3$ and complete bipartite graphs. 

Conversely, let $G$ be a graph for which equality holds. Then, by Theorem \ref{thm:Motzkin_Strauss}, the equality in \eqref{eq:AN_1} implies that
\(G\) is a complete \(\omega\)-partite graph. If \(\omega(G)\ge 3\), then \(1-2/\omega(G)>0\), so equality in \eqref{eq:AN_2} is also necessary, which implies \(x\) is constant on \(V(G)\). Since \(x\) is a \(q(G)\)-eigenvector, this implies that \(G\) is regular. Thus \(G\) is a regular complete \(\omega\)-partite graph in this case. The proof is complete.
\end{proof}

\section{Localized bound for $Q$-index}
\label{sec:localized_Q_index}

In this section, we prove Theorem \ref{thm:signless_vertex_local}. We require the following weighted Motzkin-Strauss type theorem.

\begin{theorem}[\cite{Bradac_2022, Liu_Ning_2026}]\label{thm:weighted_M_S_1}
    Let $W(G)=[w_{uv}]$ be a weighted matrix where $w_{uv} = \frac{c(uv)}{c(uv)-1}$. Then 
    \[\max_{x\in S}\ x^TW(G)x\le 1.\] Equality holds if and only if $\supp(x)$ induces a complete $\omega$-partite graph whose vertex classes $V_1, \ldots, V_{\omega}$ satisfy $\sum_{v\in V_i}x_v = \frac{1}{\omega}$ for each $1\le i\le \omega$. 
\end{theorem}

We are now ready to prove our result.

\begin{proof}[Proof of Theorem \ref{thm:signless_vertex_local}]
Let \(x>0\) be a Perron-eigenvector of \(Q(G)\) with $1$-norm equal to $1$, i.e., \[ \sum_{v\in V(G)} x_v=1. \]

Define a weighted adjacency matrix \(M\) by
\[ M_{uv}=M_{vu}:=\frac{x_u+x_v}{2\sqrt{x_u x_v}} \]
for \(uv\in E(G)\), and \(M_{uv}=0\) otherwise. Put
\[ y_v:=\sqrt{x_v}. \]

Since \(\sum_v x_v=1\), we have
\[ \sum_{v\in V(G)} y_v^2=1. \]

We claim that $\frac{q(G)}{2}$ is the largest eigenvalue of $M$ with eigenvector $y$. Observe that for every vertex \(v\in V(G)\),
\[ (My)_v = \sum_{u\sim v} \frac{x_u+x_v}{2\sqrt{x_u x_v}}\sqrt{x_u} = \frac{1}{2\sqrt{x_v}}\sum_{u\sim v}(x_u+x_v). \]

But
\[ \sum_{u\sim v}(x_u+x_v) = \sum_{u\sim v}x_u+\deg(v)x_v =(Q(G)x)_v = q(G) x_v,\]
which implies
\[ (My)_v=\frac{q(G)}{2} y_v. \]
Thus $\frac{q(G)}{2}$ is an eigenvalue of $M$ with eigenvector $y$. 

Since \(G\) is connected and \(M\) has positive weights on the edges of \(G\),
the matrix \(M\) is irreducible. Moreover, the eigenvector $y$ is non-negative. Hence, by the Perron-Frobenius Theorem,
\[ \lambda_1(M)=\frac{q(G)}{2}. \]

Now, because \(y\) is a unit Perron-eigenvector of \(M\),
\[ \lambda_1(M) = y^TMy = 2\sum_{uv\in E(G)} M_{uv}y_u y_v = 2\sum_{uv\in E(G)} \left( \sqrt{\frac{c(uv)-1}{c(uv)}}M_{uv}\right) \left(\sqrt{ \frac{c(uv)}{c(uv)-1}}y_uy_v\right). \]

Using the Cauchy-Schwarz inequality,
\[ \lambda_1(M)^2 \leq
\left(2\sum_{uv\in E(G)} \frac{c(uv)-1}{c(uv)}(M_{uv})^2\right)
\left(2\sum_{uv\in E(G)} \frac{c(uv)}{c(uv)-1}y_u^2y_v^2\right).
\]

Using Theorem \ref{thm:weighted_M_S_1},  
\begin{equation}\label{eq:signless_vertex_local_1}
    \left(2\sum_{uv\in E(G)} \frac{c(uv)}{c(uv)-1}y_u^2y_v^2\right) \le 1,
\end{equation}
which implies
\[ \lambda_1(M)^2 \leq 2\sum_{uv\in E(G)} \frac{c(uv)-1}{c(uv)}(M_{uv})^2, \]
equivalently
\[ \frac{q(G)^2}{4} \leq 2 \sum_{uv\in E(G)} \frac{c(uv)-1}{c(uv)}\left(\frac{x_u+x_v}{2\sqrt{x_u x_v}}\right)^2.\]

Therefore, 
\begin{align*}
q(G)^2 & \leq 2 \sum_{uv\in E(G)} \frac{c(uv)-1}{c(uv)}\left(\frac{x_u+x_v}{\sqrt{x_u x_v}}\right)^2\\
& = 2\sum_{uv\in E(G)} \frac{c(uv)-1}{c(uv)}
\left(\frac{x_u}{x_v} + \frac{x_v}{x_u} + 2\right) \\
& = 2\sum_{v\in V(G)} \sum_{u\sim v} \frac{c(uv)-1}{c(uv)}\left(1+\frac{x_u}{x_v}\right)\\
& \le  2 \sum_{v\in V(G)} \sum_{u\sim v} \frac{c(v)-1}{c(v)} \left(1+\frac{x_u}{x_v}\right),
\end{align*}
where the last inequality holds since $c(uv)\le \min\{c(u), c(v)\}$. 

Note that for each vertex \(v\),
\[ \sum_{u\sim v} \left(1+\frac{x_u}{x_v}\right) = \frac1{x_v}\sum_{u\sim v}(x_v+x_u) = \frac{(Q(G)x)_v}{x_v} = q(G).\]

Therefore
\[ q(G)^2 \leq 2 \sum_{v\in V(G)} \frac{c(v)-1}{c(v)} q(G), \]
which implies
\[ q(G) \leq 2\sum_{v\in V(G)} \frac{c(v)-1}{c(v)}.\]

Suppose now that $G$ is a graph for which the inequality is tight. Then in equation \eqref{eq:signless_vertex_local_1} equality holds. By Theorem \ref{thm:weighted_M_S_1},  $\supp(y)$ must induce a complete $\omega$-partite graph. Since $\supp(y) = \supp(x) = V(G)$, we conclude that $G$ is a complete $\omega$-partite graph, say with parts $V_1,\dots,V_\omega$, and
\[
        \sum_{v\in V_i}x_v=\frac1{\omega}
        \qquad (1\le i\le \omega).
\]
It is clear that the eigenvector $x$ is constant on each $V_i$, i.e., for $v\in V_i$, $x_v = \frac{1}{\omega|V_i|}$. Writing the eigen-equation for $q(G)$ gives
\[    q(G)x_v 
    = (n-|V_i|)x_v+\sum_{j\ne i} \sum_{u\in V_j}x_u  \\
    = (n-|V_i|)x_v+\frac{\omega-1}{\omega}. 
\]
Dividing by $x_v$ and using the fact that $x_v = \frac{1}{\omega |V_i|}$ gives
\[
      q(G) 
    = (n-|V_i|)+\left(\frac{\omega-1}{\omega}\right)\omega|V_i|\\
    = n + (\omega - 2)|V_i|.  
\]

If \(\omega\ge 3\), the above identity forces that all $V_i$'s have the same size, i.e., \(G\) is a regular complete \(\omega\)-partite graph. 

Conversely, it is clear that equality holds for regular $\omega$-partite graphs when $\omega \ge 3$ and complete bipartite graphs. The proof is complete.
\end{proof}

We conclude this section with a proof of Corollary \ref{cor:A_alpha_local}. 

\begin{proof}[Proof of Corollary \ref{cor:A_alpha_local}]
 For $\alpha \in [0, \frac{1}{2}]$, note that 
 \[ A_\alpha(G) = \alpha \left(D(G) + A(G)\right) + (1-2\alpha)A(G).\]
Using the well-known Weyl's inequality, we have 
\[ \lambda_1(A_\alpha(G))\le \alpha q(G) + (1-2\alpha)\lambda_1(G).\]
Applying Theorems \ref{thm:spectral_vertex_localization_Turan} and \ref{thm:signless_vertex_local}, we get the desired inequality.
\end{proof}

\section{Vertex-weighted signed graphs}\label{sec:weighted_signed_graph}

A \emph{signed graph} $\Gamma$ is a pair $(G, \sigma)$, where $G =(V(G), E(G))$ is an undirected graph, called the \emph{underlying unsigned graph}, and $\sigma:E(G)\rightarrow\{-1, +1\}$ is the \emph{sign function}. The \emph{adjacency matrix} of $\Gamma$, denoted by $A(\Gamma)$, is defined as
\[A(\Gamma)_{uv}=
\begin{cases}
    \sigma(uv)&\text{if} \mbox{ $u\sim v$},\\
    0&\text{otherwise.}
\end{cases}\]
    

The \textit{sign of a cycle} (with some orientation) $C = v_1 v_2 \ldots v_k v_1$, denoted by $\sigma(C)$, is defined as the product of the signs of its edges, that is 
$$\sigma(C) := \sigma(e_{12}) \sigma(e_{23}) \cdots \sigma(e_{(k-1)l}) \sigma(e_{k1}),$$ where $e_{ij}$ is the edge between the vertices $v_i$ and $v_j$. A cycle $C$ is  \textit{neutral} if $\sigma(C) = 1$, and a signed graph is  \textit{balanced} if all its cycles are neutral. The \emph{balanced clique number} of a signed graph $\Gamma$, denoted by $\omega_b(\Gamma)$, is the maximum number of vertices in a balanced complete subgraph of $\Gamma$. 

A function from the vertex set of $G$ to the set $\{1,-1\}$ is called a \textit{switching function}. Two signed graphs $\Gamma_1 = (G, \sigma_1)$ and $\Gamma_2 = (G, \sigma_2)$ are \textit{switching equivalent}, written as $\Sigma_1 \sim \Sigma_2$, if there is a switching function $\eta: V(G) \to \{1 ,-1\}$ such that $$\sigma_2(e_{ij})=\eta(v_i)^{-1}\sigma_1(e_{ij})\eta(v_j).$$
The switching equivalence of two signed graphs can be defined in the following equivalent way: two signed graphs $\Gamma_1 = (G, \sigma_1)$ and $\Gamma_2 = (G, \sigma_2)$ are \textit{switching equivalent}, if there exists
a diagonal matrix $D_{\eta}$ with diagonal entries from $\{1, -1\}$ such that
\begin{align}\label{eq: switching equi}
    A(\Gamma_2) = D_{\eta}^{-1}A(\Gamma_1)D_{\eta}.
\end{align}
Switching equivalence preserves connectivity and balance. For further details on signed graphs, we refer to \cite{Harary_1955, Zaslavsky_1982, Zaslavsky_1998}.

\subsection{Known results}


Wei Wang, Zhidan Yan, and Jianguo Qian \cite{Wang_Yan_Jianguo_2021} established a Motzkin-Straus type theorem for signed graphs and, as a consequence, extended Wilf's inequality to signed graphs. Precisely, they proved the following: Let $\Gamma = (G, \sigma)$ be a signed graph with $n$ vertices, and balanced clique number $\omega_{b}(\Gamma)$. Then 
\begin{equation}
    \lambda_1(A(\Gamma)) \leq n\bigg(1-\frac{1}{\omega_b(\Gamma)}\bigg).
\end{equation}

Gaoxing Sun, Feng Liu, and Kaiyang Lan \cite{Sun_Liu_Lan_2022} obtained the following edge version, which generalizes the above and also the spectral Tur\'{a}n's Theorem:
\begin{equation}
    \lambda_1^2(A(\Gamma)) \leq 2m \left(1 - \frac{1}{\omega_b(\Gamma)}\right).
\end{equation}

Subsequently, Kannan and Pragada \cite{Kannan_Pragada_2023} proved that
\begin{equation}
\lambda_1^2(A(\Gamma)) \leq 2(m - \epsilon(\Gamma))\left(1 - \frac{1}{\omega_b(\Gamma)}\right),
\end{equation}
where $\epsilon(\Gamma)$ denotes the frustration index of a signed graph $\Gamma$, which is the minimum number of edges to remove for balance such that the signed graph is balanced. 

For a vertex $v\in V(\Gamma)$ and edge $uv \in E(\Gamma)$, we denote by $c_b(v)$ and $c_b(uv)$ the order of the largest balanced complete subgraph of $\Gamma$ containing the vertex $v$ and edge $uv$, respectively.



Linfeng Xie and Xiaogang Liu \cite{Xie_Liu_2025} established a localized Motzkin-Struass theorem for signed graphs, and as a consequence, they established the following edge-localized spectral Tur\'{a}n's theorem for signed graphs:
\begin{equation}
    \lambda_1^2(A(\Gamma)) \leq 2 \sum_{uv \in E^+(\Gamma')} 1-\frac{1}{c_b(uv)} \le 2 \sum_{uv \in E(\Gamma)} 1-\frac{1}{c_b(uv)},
\end{equation}
where $\Gamma' \sim \Gamma$ is such that $\lambda_1(\Gamma')$ has a non-negative eigenvector and $E^+(\Gamma')$ is the set of positive edges in $\Gamma'$. This result generalizes all the above bounds. 

\subsection{Our contribution}

For the signed Laplacian, even the non-localized versions are not established for signed graphs. Here, we establish a general localized result for weighted signed graphs. 

For a signed graph $\Gamma = (G, \sigma)$, denote its adjacency matrix by $A(\Gamma)$. For a weight function $w:V(G)\rightarrow \mathbb{R}^{+}$, consider the \emph{vertex weighted adjacency matrix} $W(G)$ of $G$ defined as:
\[ W(G)_{uv}:= \begin{cases}
    \sqrt{w(u)w(v)} & \text{if } uv\in E(G),\\
    0 & \text{otherwise}.
\end{cases}
\]
We denote by $D(G)$ the \emph{vertex weighted degree matrix}, which is a diagonal matrix with entry $\sum_{u \sim v} w(u)$ corresponding to vertex $v$.
Define the corresponding vertex-weighted adjacency matrix of the signed graph $\Gamma$ to be
\[W(\Gamma) := W(G)\circ A(\Gamma).\] 
 For a given $\Gamma$ and $W(G)$, we define the corresponding \emph{vertex weighted Laplacian} and \emph{vertex weighted signless Laplacian} of $\Gamma$ to be
\[L(\Gamma) := D(G) - W(\Gamma)\quad \text{and}\quad Q(\Gamma) := D(G) + W(\Gamma) = L(-\Gamma),\]
respectively. With the above notation, we prove a much stronger localization inequality for $\lambda_1(Q(\Gamma))$ of weighted signed graphs. 

\begin{theorem}\label{thm:weighted_signed_local} Let $\Gamma = (G, \sigma)$ be a connected signed graph and consider a weight function $w:V(G)\rightarrow \mathbb{R}^+$. Then
\[\lambda_1(Q(\Gamma)) \leq 2\sum_{v\in V(G)} \left(1 -\frac{1}{c_b(v)}\right) w(v).\]
\end{theorem}

If one takes $\sigma$ and $w$ to be the all $1$ functions, then the above Theorem \ref{thm:weighted_signed_local} implies Theorem \ref{thm:signless_vertex_local}.   

\begin{proof}
We assume that $n\ge 2$ so that $c_b(v)\ge 2$ for every $v\in V(\Gamma)$. Let $x$ denote a unit eigenvector corresponding to $\lambda_1(Q(\Gamma))$. Without loss of generality, we can assume that $x\ge 0$. This is because the graph $\Gamma'$ obtained from \(\Gamma\) by switching w.r.t. the signs of the entries of \(x\) has $\lambda_1(Q(\Gamma')) = \lambda_1(Q(\Gamma))$ with corresponding eigenvector $|x|$. One can then work with $\Gamma'$ instead of $\Gamma$.

We know that $\lambda_1(Q(\Gamma)) = x^T Q(\Gamma)x$. Observe that
\begin{align*}
x^TQ(\Gamma)x
&=x^T(D(G)+W(\Gamma))x \\
&=\sum_{v\in V(G)}
\left(\sum_{u\sim v}w(u)\right)x_v^2
+2\sum_{uv\in E(G)}
\sigma_{uv}\sqrt{w(u)w(v)}x_ux_v \\
&=\sum_{uv\in E(G)}
\left(w(u)x_v^2+w(v)x_u^2
+2\sigma_{uv}\sqrt{w(u)w(v)}x_ux_v\right) \\
&=\sum_{uv\in E(G)}
\left(\sqrt{w(u)}x_v+\sigma_{uv}\sqrt{w(v)}x_u\right)^2 .
\end{align*}

Let \(H\) be the spanning subgraph of \(G\) consisting of the positive edges
of \(\Gamma\). Let $y$ be the vector such that \[y_v:= \sqrt{w(v)} x_v \quad (v\in V(\Gamma)).\] 
Then
\begin{align}\label{eq:ws_1}
x^TQ(\Gamma)x 
&= \sum_{uv\in E(G)}
\left(\sqrt{w(v)}x_u + \sigma_{uv}\sqrt{w(u)}x_v\right)^2 \nonumber \\
&= 
\sum_{uv\in E(H)}
\left(\sqrt{w(v)}x_u+\sqrt{w(u)}x_v\right)^2 +
\sum_{uv\in E(G)\setminus E(H)}
\left(\sqrt{w(v)}x_u-\sqrt{w(u)}x_v\right)^2 \nonumber \\
&\le
\sum_{uv\in E(K_n)}
\left(\sqrt{w(v)}x_u-\sqrt{w(u)}x_v\right)^2
+
4\sum_{uv\in E(H)}\sqrt{w(u)w(v)}x_ux_v \nonumber \\
&=
\left(\sum_{v\in V(G)} w(v)\right)\left(\sum_{v\in V(G)} x_v^2\right)-\left(\sum_{v \in V(G)}\sqrt{w(v)}x_v \right)^2+4\sum_{uv\in E(H)}\sqrt{w(u)w(v)}x_ux_v \nonumber \\
&= 
\left(\sum_{v\in V(G)} w(v)\right)+4\sum_{uv\in E(H)}y_uy_v -\left(\sum_{v \in V(G)}y_v \right)^2.
\end{align}

By Theorem \ref{thm:weighted_M_S_1} applied to the all-positive
graph \(H\) with vector \(y_v/\sum_{v \in V(G)} y_v\), we get
\[2\sum_{uv\in E(H)} \frac{c_H(uv)}{c_H(uv)-1}y_uy_v \le \left(\sum_{v \in V(G)}y_v \right)^2,
\]
where $c_H(uv)$ denotes the size of the largest clique in $H$ containing the edge $uv$. Hence
\begin{equation}\label{eq:ws_2}
4\sum_{uv\in E(H)}y_uy_v-\left(\sum_{v \in V(G)}y_v \right)^2
\le
2\sum_{uv\in E(H)}
\frac{c_H(uv)-2}{c_H(uv)-1}y_uy_v.
\end{equation}

For any vertex $v\in V(H) = V(\Gamma)$, define
\[
\delta_v=
\begin{cases}
1-\dfrac{2}{c_H(v)} & \text{ if }v \text{ is incident to an edge of } H,\\[6pt]
0 & \text{ if } v \text{ is isolated in } H,
\end{cases}
\]
where $c_H(v)$ denotes the size of the largest clique in $H$ containing the vertex $v$. We claim that,
\begin{equation}\label{eq:ws_3}
2\sum_{uv\in E(H)}
\frac{c_H(uv)-2}{c_H(uv)-1}
\sqrt{w(u)w(v)}\,x_ux_v
\le
\left(\sum_{v \in V(G)} \delta_v w(v)\right).
\end{equation}

Using the Cauchy-Schwarz inequality, we have 
\begin{align*}
& \quad \ \left(
2\sum_{uv\in E(H)} \frac{c_H(uv)-2}{c_H(uv)-1} \sqrt{w(u)w(v)} \ x_ux_v \right)^2 \\
&=  \left(2\sum_{uv\in E(H)} \sqrt{\frac{c_H(uv)}{c_H(uv)-1}}x_ux_v \cdot\frac{c_H(uv)-2}{\sqrt{c_H(uv)(c_H(uv)-1)}} \sqrt{w(u)w(v)}\right)^2\\
&\le
\left( 2\sum_{uv\in E(H)} \frac{c_H(uv)}{c_H(uv)-1}x_u^2x_v^2 \right) \left( 2\sum_{uv\in E(H)} \frac{(c_H(uv)-2)^2}{c_H(uv)(c_H(uv)-1)}w(u)w(v) \right) \\
&\le
2\sum_{uv\in E(H)} \frac{(c_H(uv)-2)^2}{c_H(uv)(c_H(uv)-1)}w(u)w(v),
\end{align*}
where the last inequality follows by Theorem \ref{thm:weighted_M_S_1}. Since \(c_H(uv)\le \min\{c_H(u),c_H(v)\}\), we have
\[ \min\{\delta_u,\delta_v\}\ge \frac{c_H(uv)-2}{c_H(uv)},\]
which implies
\begin{align*}
& \quad \ 2\sum_{uv\in E(H)}
\frac{(c_H(uv)-2)^2}{c_H(uv)(c_H(uv)-1)}w(u)w(v) \\
& \le 
2\sum_{uv\in E(H)} \frac{c_H(uv)}{c_H(uv)-1}\delta_u\delta_v w(u)w(v) \\
& \le
\left(\sum_{v \in V(G)} \delta_v w(v)\right)^2,
\end{align*}
where the last inequality follows by Theorem \ref{thm:weighted_M_S_1}. This proves the inequality \eqref{eq:ws_3}. Combining  the inequalities \eqref{eq:ws_1}, \eqref{eq:ws_2} and \eqref{eq:ws_3}, we get
\[ \lambda_1(Q(\Gamma)) = x^TQ(\Gamma)x \le \left(\sum_{v\in V(G)} w(v)\right)+\left(\sum_{v \in V(G)} \delta_v w(v)\right). \]

Now, every clique of \(H\) is an all-positive clique of $\Gamma$, implying \(c_H(v)\le c_b(v)\) and consequently
\[ \delta_v\le 1-\frac{2}{c_b(v)}. \]
Hence
\[\lambda_1(Q(\Gamma))\le \sum_{v\in V(G)} w(v)+\sum_{v\in V(G)}\left(1-\frac{2}{c_b(v)}\right)w(v) = 2\sum_{v\in V(G)}\left(1-\frac{1}{c_b(v)}\right)w(v). \]
This completes the proof. 
\end{proof}

\section{Open problems}
\label{sec:open_problems}

In this paper, we focused on vertex localization
of Tur\'{a}n type inequalities for $q(G)$. Yongtao Li, Weijun Liu, and Lihua Feng in the survey paper \cite{Li_Liu_Feng_2022} asked whether the following edge spectral version of Turan's theorem holds for $q(G)$: 
\[q(G)^2 \le 8m \left(1 - \frac{1}{\omega(G)}\right).\]
This bound doesn't hold as they demonstrate in their paper. We propose the following edge localized inequality for $q(G)$.

\begin{conjecture}\label{conj:signless_edge_local} For any connected graph $G$, we have 
\begin{equation}\label{eq:signless_edge_local}
    q(G)\le 2\sum_{uv\in E(G)}\left(1-\frac{1}{c(uv)}\right)\left(\frac{1}{\deg(u)} + \frac{1}{\deg(v)}\right).
\end{equation}
\end{conjecture}

\begin{remark}
Note that 
\begin{align*}
    & \quad \ 2\sum_{uv\in E(G)}\left(1-\frac{1}{c(uv)}\right)\left(\frac{1}{\deg(u)} + \frac{1}{\deg(v)}\right) \\
    & \le 2\sum_{uv\in E(G)}\left(1-\frac{1}{c(u)}\right)\frac{1}{\deg(u)} + \left(1-\frac{1}{c(v)}\right)\frac{1}{\deg(v)}\\
    & = 2 \sum_{v\in V(G)} \left(1-\frac{1}{c(v)}\right). 
\end{align*}
Hence, Conjecture \ref{conj:signless_edge_local}, if true, implies Theorem \ref{thm:signless_vertex_local}.  
\end{remark} 

The corresponding signed analogue of Conjecture \ref{conj:signless_edge_local}, with \(c(uv)\) replaced by \(c_b(uv)\), is false as we show below.

For $n\ge 4$, let \(\Gamma_n\)  be the signed complete graph \(K_n\) with $V(K_n) = \{1, 2, \ldots, n\}$, and exactly one negative edge, say \(12\), and all other edges positive. It is clear that 
\[ c_b(e) = 
\begin{cases}
    n-1 & \text{ if }e\neq 12,\\
    2 & \text{ if }e = 12.
\end{cases}\]

Therefore the RHS in Conjecture \ref{conj:signless_edge_local} evaluates to 
\begin{equation}\label{eq:signed_counterexample_1}
  2(n-1)-\frac{4(n-2)}{(n-1)^2}.  
\end{equation}

On the other hand,
\[ Q(\Gamma_n)=(n-1)I+A(\Gamma_n).\]
The quotient matrix for $A(\Gamma_n)$ is given by 
\[
\begin{pmatrix}
        -1 & n-2 \\
        2 & n-3
\end{pmatrix},
\]
which has eigenvalues 
\[
        \frac{n-4\pm\sqrt{n^2+4n-12}}{2}.
\]
Thus, for $n\ge 4$, we have
\begin{align}\label{eq:signed_counterexample_2}
\lambda_1(Q(\Gamma_n)) & = (n-1) + \lambda_1(A(\Gamma_n))\nonumber\\
& = (n-1) + \frac{n-4+\sqrt{n^2+4n-12}}{2} \nonumber\\
& = \frac{3n-6+\sqrt{n^2+4n-12}}{2}.
\end{align}

Using an elementary calculation, it can be checked that the expression in \eqref{eq:signed_counterexample_1} is strictly less than $\lambda_1(Q(\Gamma_n))$ given in \eqref{eq:signed_counterexample_2}. Hence, the signed analogue of Conjecture \ref{conj:signless_edge_local} fails for $\Gamma_n \ (n\ge 4)$. 

\section*{Acknowledgements}

We acknowledge the use of ChatGPT 5.5 Pro during the ideation phase. We declare that the text is not AI-generated. M. Rajesh Kannan acknowledges financial support from the ANRF-CRG India. 

\bibliographystyle{plain}
\bibliography{signless_references}

\vspace{0.4cm}

\affl{M. Rajesh Kannan}{rajeshkannan@math.iith.ac.in, rajeshkannan1.m@gmail.com}{Department of Mathematics, Indian Institute of Technology Hyderabad, Sangareddy 502285, India}

\affl{Hitesh Kumar}{hitesh.kumar.math@gmail.com, hitesh\_kumar@sfu.ca}{Department of Mathematics, Simon Fraser University, Burnaby, Canada}

\affl{Shivaramakrishna Pragada}{shivaramakrishna\_pragada@sfu.ca, shivaramkratos@gmail.com}{Department of Mathematics, Simon Fraser University, Burnaby, Canada}

\end{document}